\documentclass[a4paper]{article}
\usepackage{amsmath}
\usepackage{amsthm}
\usepackage{amsfonts}
\usepackage{amsbsy}

\def\rmd{\mathrm{d}}
\def\R{\mathbb{R}}
\providecommand{\abs}[1]{\left\lvert#1\right\rvert}
\providecommand{\set}[1]{\left\{#1\right\}}
\providecommand{\vect}[1]{\mathbf{#1}}
\def\cover{{\mathcal{N}}}
\def\pack{{\mathcal{P}}}
\def\diamcover{{\mathcal{D}}}

\DeclareMathOperator\diameter{diam}
\DeclareMathOperator\gen{gen}

\newtheorem{definition}{Definition}[section]
\newtheorem{lemma}[definition]{Lemma}
\newtheorem{theorem}[definition]{Theorem}
\newtheorem{corollary}[definition]{Corollary}
\newtheorem{example}[definition]{Example}

\begin{document}

  \title{Generalised Cantor sets and the dimension of products}
\author{Eric J. Olson\footnote{Mathematics and Statistics,
		University of Nevada,
		Reno, NV, \textup{89507}, USA. \texttt{ejo@unr.edu}}, James C. Robinson\footnote{Mathematics Institute, University of Warwick, Coventry, UK, CV\textup{4} \textup{7}AL. \texttt{j.c.robinson@warwick.ac.uk}}, and Nicholas Sharples\footnote{Department of Mathematics, Imperial College London, London, UK, SW\textup{7} \textup{2}AZ. \texttt{nicholas.sharples@gmail.com}}}  		  

\maketitle

\begin{abstract}
In this paper we consider the relationship between the Assouad and box-counting dimension and how both behave under the operation of taking products. We introduce the notion of `equi-homogeneity' of a set, which requires a uniformity in the size of local covers at all lengths and at all points. We prove that the Assouad and box-counting dimensions coincide for sets that have equal upper and lower box-counting dimensions provided that the set `attains' these dimensions (analogous to `$s$-sets' when considering the Hausdorff dimension), and the set is equi-homogeneous. Using this fact we show that for any $\alpha\in(0,1)$ and any $\beta,\gamma\in(0,1)$ such that $\beta+\gamma\ge1$ we can construct two generalised Cantor sets $C$ and $D$ such that $\dim_B C =\alpha \beta$, $\dim_B D =\alpha \gamma$, and $\dim_A C=\dim_A D=\dim_A(C\times D)=\dim_B(C\times D)=\alpha$.
\end{abstract}

\section{Introduction}

In this paper we study the behaviour of the box-counting and Assouad
dimensions (whose definitions we give below) under the action of
taking the Cartesian product of sets. Relatively straightforward
arguments can be used to show that the Assouad and upper box-counting dimensions satisfy
\[
\dim(A\times B)\le\dim A+\dim B,
\]
but constructing examples showing that this inequality is strict is
less straightforward. For the box-counting dimension, the first
example of sets for which there is strict inequality was constructed
by Robinson \& Sharples in \cite{RobinsonSharples13RAEX}: these are
Cantor-like sets with carefully controlled ratios, much as those in
this paper. A significantly simpler example involving two countable
sets followed later from Olson \& Robinson \cite{ORbd}. For the
Assouad dimension, there is an example of strict inequality due to
Larman \cite{Larman} (see also Section 9.2 in Robinson \cite{JCR}) of
two subsets of $\R$ that accumulate at zero in such a way that the
sets and their product all have dimension one.

In this paper we provide a unified treatment of the two dimensions
using `generalised Cantor sets', i.e.\ Cantor sets in which we allow
the portion removed to vary at each stage of the construction in a controlled way.
Our argument to calculate the Assouad dimension of generalised Cantor sets and their products relies on the `equi-homogeneity' of these sets (defined below): roughly this is the property that the range of the number of balls required in the `local covers' of the set is uniformly bounded at all length-scales.
We discuss equi-homogeneity in a more general setting in \cite{HORS} where we prove that the attractors of a large class of iterated function systems are equi-homogeneous. However, the arguments presented here will serve as prototypes for the more general results in \cite{HORS}.

\subsection{Counting covers}
We begin by defining some notions of dimension for subsets of a metric space $\left(X,\rmd_{X}\right)$. We adopt the notation $B_{\delta}\left(x\right)$ for the closed ball of radius $\delta$ with centre $x\in X$, and for brevity we refer to sets of this form as \textit{$\delta$-balls}. For a set $F\subset X$ and a length $\delta>0$ we denote by $\cover(F,\delta)$ the minimum number of $\delta$-balls such that $F$ is contained in their union. If $\cover\left(F,\delta\right)$ is finite for all $\delta>0$ we say that the set $F$ is \emph{totally bounded}.
We recall that for each $\delta>0$ the function $\cover(\cdot,\delta)$ is
\begin{itemize}
\item monotonic, that is $A\subset B \Rightarrow \cover(A,\delta)\leq
  \cover(B,\delta)$, and
\item subadditive, that is $\cover(A\cup B,\delta)\leq
  \cover(A,\delta)+\cover(B,\delta)$,
\end{itemize}
and that for each set $F\subset X$ the function $\cover(F,\cdot)$ is non-increasing.

There are many similar geometric quantities, some of which we will
make use of in what follows:
\begin{itemize}
\item $\diamcover(F,\delta)$, the minimum number of sets of diameter $\delta$ that cover $F$, where the diameter of a set
$A$ is given by $\diameter(A)=\sup\set{\abs{x-y}:\ x,y\in A}$;
\item $\pack(F,\delta)$, the maximum number of disjoint $\delta$-balls with centres in $F$.
\end{itemize}
It is a short exercise to establish that these geometric quantities satisfy
\begin{align}\label{geometric inequalities}
  \diamcover(F,4\delta)&\leq \cover\left(F,2\delta\right) \leq \pack\left(F,\delta\right)
  \leq \diamcover(F,\delta)
\end{align}
(see, for example, Definitions 3.1 in Falconer \cite{BkFalconer03} or
Lemma 2.1 in Robinson \& Sharples \cite{RobinsonSharples13RAEX}).

We adopt the cover by $\delta$-balls as our primary measure since it is convenient for sets of the form $B_{\delta}\left(x\right)\cap F$, which feature in the definition of the Assouad dimension.

\subsection{Box-Counting Dimension}

First, we recall the definition of the familiar box-counting
dimensions.

\begin{definition}\label{definition - box-counting dimensions}
  For a totally bounded set $F\subset X$ we define the lower and upper
  box-counting dimensions of $F$ as the quantities
  \begin{align*}
    \dim_{\rm LB}F&:= \liminf_{\delta\rightarrow 0+} \frac{\log \cover(F,\delta)}{-\log \delta},\\
    \text{and}\qquad \dim_{\rm B}F&:= \limsup_{\delta\rightarrow 0+}
    \frac{\log \cover(F,\delta)}{-\log \delta}
  \end{align*}
  respectively.
\end{definition}

In light of the inequalities \eqref{geometric inequalities}, replacing $\cover\left(F,\delta\right)$ with any of the geometric quantities mentioned above gives an equivalent definition.
The box-counting dimensions essentially capture the exponent
$s\in\R^{+}$ for which the minimum number of $\delta$-balls
required to cover $F$ scales like $\cover(F,\delta)\sim \delta^{-s}$. More
precisely, it follows from Definition \ref{definition - box-counting dimensions} that for all $\delta_{0}>0$ and any $\varepsilon>0$ there exists a constant $C\geq 1$ such that
\begin{align}
  C^{-1}\delta^{-\dim_{\rm LB}F+\varepsilon} &\leq
  \cover(F,\delta)\leq C\delta^{-\dim_{\rm B}F-\varepsilon} &
  \text{for all}\; 0<\delta\leq\delta_0.\label{box-counting growth bounds}
\end{align}

In some cases the bounds \eqref{box-counting growth bounds} will also hold at the limit $\varepsilon\rightarrow 0$, that is for each $\delta_{0}>0$ there exists a constant $C\geq 1$ such that
\begin{equation}\label{attains}
\frac{1}{C} \delta^{-\dim_{\rm LB}F}\leq \cover\left(F,\delta\right)\leq C\delta^{-\dim_{\rm B}F}\qquad \text{for all} \;0<\delta\leq \delta_{0},
\end{equation}
giving precise control of the growth of $\cover(F,\delta)$.
We distinguish this class of sets in the following definition:
\begin{definition}\label{boxattain}
  We say that a bounded set $F\subset X$ \emph{attains} its lower
  box-counting dimension if for all $\delta_{0}>0$ there exists a positive constant $C\leq 1$ such that
  \begin{align*}
    \cover(F,\delta)&\geq C\delta^{-\dim_{\rm LB}F} &\mbox{for all}\quad
    0<\delta\leq\delta_0.
\intertext{Similarly, we say that $F$ \emph{attains} its upper box-counting dimension if for all $\delta_{0}>0$ there exists a constant $C\geq 1$ such that}
    \cover(F,\delta)&\leq C\delta^{-\dim_{\rm B}F} &\mbox{for all}\quad
    0<\delta\leq\delta_0.
  \end{align*}
\end{definition}

We remark that a similar distinction is made with regard to the
Hausdorff dimension of sets: recall that the Hausdorff measures are a
one-parameter family of measures, denoted $\mathcal{H}^{s}$ with
parameter $s\in\R^{+}$, and that for each set $F\subset \R^{n}$ there
exists a value $\dim_{\rm H}F\in\R^{+}$, called the Hausdorff dimension of
$F$, such that
\begin{equation*}
  \mathcal{H}^{s}\left(F\right)=
  \begin{cases}
    \infty & s<\dim_{\rm H}F\\
    0 & s>\dim_{\rm H}F.
  \end{cases}
\end{equation*}
For a set $F$ to have Hausdorff dimension $d$ it is sufficient, but
not necessary, for the Hausdorff measure with parameter $d$ to satisfy
$0<\mathcal{H}^{d}\left(F\right)<\infty$. Sets with this property are
sometimes called $d$-sets (see, for example, \cite{BkFalconer03}
pp.32) and are distinguished as they have many convenient
properties. For example, the Hausdorff dimension product formula
$\dim_{\rm H}\left(F\times G\right)\geq \dim_{\rm H}F+\dim_{\rm H}G$ was first
proved for sets $F$ and $G$ in this restricted class (see Besicovitch
\& Moran \cite{BesicovitchMoran45}) before being extended to hold for
all sets (see Howroyd \cite{Howroyd95}).

\subsection{Homogeneity and the Assouad Dimension}\label{section - intro homogeneity}

The Assouad dimension is a less familiar notion of dimension, in which
we are concerned with `local' coverings of a set $F$: for more details
see Assouad \cite{Assouad}, Bouligand \cite{Bouligand}, Fraser \cite{Fraser2014}, Luukkainen
\cite{Luuk}, Olson \cite{EJO}, or Robinson \cite{JCR}.

\begin{definition}\label{definition - homogeneous}
  A set $F\subset X$ is $s$-homogeneous if for all $\delta_{0}>0$ there exists a constant
  $C>0$ such that
  \begin{equation*}
    \cover(B_{\delta}\left(x)\cap F,\rho\right)\leq C\left(\delta/\rho\right)^{s}\qquad \forall\;x\in F, \quad \text{for all}\ \delta,\rho\ \text{with}\ 0<\rho<\delta\leq \delta_{0}.
  \end{equation*}
\end{definition}

Note that we do not require a set to be totally bounded in order for it to be $s$-homogeneous.
\begin{definition}
  The Assouad dimension of a set $F\subset \R^{n}$ is defined by
 \[
    \dim_{\rm A} F := \inf\set{s\in\R^{+}:\ F\mbox{ is
        $s$-homogeneous}}
 \]
\end{definition}
It is known that for a totally bounded set $F\subset X$ the three notions of
dimension that we have now introduced satisfy
\begin{align}\label{dimension inequalities}
  \dim_{\rm LB}F\leq\dim_{\rm B}F\leq\dim_{\rm A}F
\end{align}
(see, for example, Lemma 9.6 in Robinson \cite{JCR} or Lemma 1.9 of \cite{HORS}).  An interesting
example is given by the compact countable set
$F_{\alpha}:=\set{n^{-\alpha}}_{n\in\mathbb{N}}\cup\set{0}\subset \R$ with
$\alpha>0$ for which
\begin{align*}
  \dim_{\rm LB}F_{\alpha}&=\dim_{\rm B}F_{\alpha}=\left(1+\alpha\right)^{-1}\\
  \text{but}\qquad\dim_{\rm A}F_{\alpha}&=1.
\end{align*}
(see Olson \cite{EJO} and Example 13.4 in Robinson \cite{JCRidds}).

\subsection{Product Sets}
Let $\left(X,\rmd_{X}\right)$ and $\left(Y,\rmd_{Y}\right)$ be metric spaces and endow the product space $X\times Y$ with a metric $\rmd_{X\times Y}$ that satisfies
\begin{align}\label{product metric condition}
m_{1}\max\left(\rmd_{X},\rmd_{Y}\right) \leq \rmd_{X\times Y} \leq m_{2} \max\left(\rmd_{X},\rmd_{Y}\right)
\end{align}
for some $m_{1},m_{2}>0$ with $m_{1}\leq m_{2}$. Clearly the familiar product metric
\begin{align*}
\rmd_{X\times Y,\infty}:&=\max\left(\rmd_{X},\rmd_{Y}\right)
\intertext{satisfies \eqref{product metric condition}, as do the metrics}
\rmd_{X\times Y,p}:&= \left(\rmd_{X}^{p}+\rmd_{Y}^{p}\right)^{\frac{1}{p}} &\text{for}\;p\in\left[1,\infty\right)
\end{align*}
with $m_{1}=1$ and $m_{2}=2^{\frac{1}{p}}$.

It is well known that if $F\subset X$ and
$G\subset Y$ are two totally bounded sets then the box-counting and
Assouad dimensions of their product $F\times G\subset X\times Y$ satisfy
\begin{align}
  \dim_{\rm LB}\left(F\times G\right) &\geq \dim_{\rm LB}F+\dim_{\rm LB}G\label{lower box product formula}\\
  \dim_{\rm B}\left(F\times G\right) &\leq \dim_{\rm B}F+\dim_{\rm B}G\label{upper box product formula}\\
  \text{and}\qquad\dim_{\rm A}\left(F\times G\right) &\leq
  \dim_{\rm A}F+\dim_{\rm A}G.\label{Assouad product formula}
\end{align}
provided that the product metric $\rmd_{X\times Y}$ satisfies \eqref{product metric condition}.

The box-counting dimension product formulae were improved in Robinson \& Sharples \cite{RobinsonSharples13RAEX}
who demonstrate that product sets satisfy the chain of
inequalities
\begin{multline}\label{product inequality chain}
  \dim_{\rm LB}F+\dim_{\rm LB}G\leq \dim_{\rm LB}\left(F\times G\right)\\
  \leq \min\set{\dim_{\rm LB}F+\dim_{\rm B}G,\dim_{\rm B}F+\dim_{\rm LB}G}\\
  \leq \max\set{\dim_{\rm LB}F+\dim_{\rm B}G,\dim_{\rm B}F+\dim_{\rm LB}G}\\
  \leq \dim_{\rm B}\left(F\times G\right)\leq \dim_{\rm B}F+\dim_{\rm B}G,
\end{multline}
and that paper provides a method for constructing sets so that their box-counting
dimensions can take arbitrary values satisfying this chain of
inequalities.

We remark that if $\dim_{\rm LB}F=\dim_{\rm B}F$ then it follows from \eqref{product inequality chain} that there is equality in \eqref{lower box product formula} and \eqref{upper box product formula}, so the good behaviour of just one set guarantees equality in the box-counting product formulas.

The box-counting dimension product formulae in \eqref{product inequality chain} are all consequences of the
geometric inequalities
\begin{align*}
  \cover\big(F\times G,m_{2}\delta\big)&\leq \cover(F,\delta)\cover(G,\delta)\\
  \text{and}\qquad \pack\left(F\times G,m_{1}\delta\right)&\geq
  \pack\left(F,\delta\right)\pack\left(G,\delta\right),
\end{align*}
which in turn follow from the inclusions
\begin{align}\label{product ball inclusions}
B_{\delta/m_{2}}\left(x\right)\times B_{\delta/m_{2}}\left(y\right) \subset B_{\delta}\left(\left(x,y\right)\right) \subset B_{\delta/m_{1}}\left(x\right)\times B_{\delta/m_{1}}\left(y\right),
\end{align}
as the product of $\delta$-ball covers of $F$ and $G$ gives rise to an $m_{2}\delta$-ball cover of $F\times G$, and the product of disjoint $\delta$-balls with centres in $F$ and $G$ gives rise to a set of disjoint $m_{1}\delta$-balls with centres in $F\times G$ (see, for example, Falconer \cite{BkFalconer03} or Robinson \& Sharples
\cite{RobinsonSharples13RAEX} for further details).
Combining these product inequalities with the relationships in \eqref{geometric inequalities} we obtain the expression
\begin{align}\label{product geometry inequality in N}
  \cover(F,4\delta/m_{1})\cover(G,4\delta/m_{1})\leq \cover(F\times G,\delta)\leq
  \cover(F,\delta/m_{2})\cover(G,\delta/m_{2})
\end{align}
which will be useful in the remainder.

To establish the Assouad dimension product inequality we prove the
following geometric relationship. One can find a very similar argument
for both bounds in Olson \cite{EJO} (Theorem 3.2) and for the upper
bound in Robinson \cite{JCR} (Lemma 9.7).

\begin{lemma}\label{lemma - Assouad product geometry}
  If $F\subset X$ and $G\subset Y$ then for all
  $\vect{x}=\left(x,y\right)\in F\times G$ and all $\delta,\rho>0$
  \begin{align*}
    \cover\left(B_{\delta}(\vect{x})\cap \left(F\times
        G\right),\rho\right)&\leq \cover\left(B_{\delta/m_{1}}(x)\cap
      F,\rho/m_{2}\right)\cover\left(B_{\delta/m_{1}}(y)\cap
      G,\rho/m_{2}\right) \intertext{and}
    \cover\left(B_{\delta}(\vect{x})\cap \left(F\times
        G\right),\rho\right)&\geq \cover\left(B_{\delta/m_{2}}(x)\cap
      F,4\rho/m_{1}\right)\cover\left(B_{\delta/m_{2}}(y)\cap
      G,4\rho/m_{1}\right).
  \end{align*}
\end{lemma}

\begin{proof}
  From \eqref{product ball inclusions} it follows that
  \begin{align*}
    B_{\delta}(\vect{x})\cap(F\times G)&\subset \left(B_{\delta/m_{1}}(x)\cap F\right)\times\left(B_{\delta/m_{1}}(y)\cap G\right)\\
    \text{and}\qquad B_{\delta}(\vect{x})\cap(F\times G)&\supset \left(B_{\delta/m_{2}}(x)\cap F\right)\times\left( B_{\delta/m_{2}}(y)\cap G\right).
  \end{align*}
Consequently, as the function $\cover\left(\cdot,\rho\right)$ is monotonic, it follows from \eqref{product geometry inequality in N} that
\begin{align*}
\cover\left(B_{\delta}(\vect{x})\cap \left(F\times
        G\right),\rho\right)&\leq \cover\left(\left(B_{\delta/m_{1}}(x)\cap F\right)\times\left(B_{\delta/m_{1}}(y)\cap G\right),\rho\right)\\
        &\leq \cover\left(B_{\delta/m_{1}}(x)\cap F,\rho/m_{2}\right)\cover\left(B_{\delta/m_{1}}(y)\cap G,\rho/m_{2}\right),
        \intertext{and}
\cover\left(B_{\delta}(\vect{x})\cap \left(F\times G\right),\rho\right)&\geq \cover\left(\left(B_{\delta/m_{2}}(x)\cap F\right)\times\left( B_{\delta/m_{2}}(y)\cap G\right),\rho\right)\\
&\geq \cover\left(B_{\delta/m_{2}}(x)\cap F,4\rho/m_{1}\right)\cover\left(B_{\delta/m_{2}}(y)\cap G,4\rho/m_{1}\right)
\end{align*}
as required.
\end{proof}

It is now simple to prove the following Assouad dimension formula for
products. We remark that in Olson \cite{EJO}, Theorem 3.2, it was
mistakenly asserted that equality holds in this product
formula. However, the argument there (which we reproduce here) shows that equality \emph{does} hold for
products of the form $F\times F$.

\begin{lemma}\label{Adimprod}
  If $F\subset X$ and $G\subset Y$ then
\begin{equation}\label{Adimprodeq}
  \dim_A(F\times G)\le\dim_A F+\dim_A G
\end{equation}
  and
  \begin{equation}\label{dAFF}
    \dim_A(F\times F)=2\dim_A F.
  \end{equation}
\end{lemma}

\begin{proof}
 Fix $\delta_{0}>0$. If $F$ is an $s$-homogeneous set and $G$ is a $t$-homogeneous set
  then from Lemma \ref{lemma - Assouad product geometry} it follows that for
  all $\delta,\rho$ with $0<\rho<\delta\leq\delta_{0}$
  \begin{align*}
    \cover(B_{\delta}\left(\vect{x})\cap \left(F\times
        G\right),\rho\right)&\leq \cover\left(B_{\delta/m_{1}}(x)\cap
      F,\rho/m_{2}\right)\cover\left(B_{\delta}(y)\cap
      G,\rho/m_{2}\right). \intertext{Therefore, since the sets $F$ and $G$
      are homogeneous and $0<\rho/m_{2}<\delta/m_{1}\leq \delta_{0}/m_{1}$, there exist constants $C_{F},C_{G}>0$ so that}
    &\leq C_{F}C_{G}\left(\frac{\delta/m_{1}}{\rho/m_{2}}\right)^{s}\left(\frac{\delta/m_{1}}{\rho/m_{2}}\right)^{t}\\
    &\leq C_{F}C_{G}\left(m_{2}/m_{1}\right)^{s+t}\left(\delta/\rho\right)^{s+t}.
  \end{align*}
  As $\delta_{0}>0$ was arbitrary we conclude that the set $F\times G$ is
  $\left(s+t\right)$-homogeneous, from which we obtain \eqref{Adimprodeq}.

  Now suppose that $F=G$. Given $\epsilon>0$, find $x\in F$ such that
\[
\cover(B_\delta(x)\cap F,\rho)\ge C(\delta/\rho)^{s-\epsilon}
\]
for some $0<\rho<\delta$. Then for $\vect{x}=(x,x)\in F\times
F$ we have
\begin{align*}
  \cover\left(B_{m_{2}\delta}(\vect{x})\cap \left(F\times
      F\right),m_{1}\rho/4\right)
  &\geq \cover \left(B_\delta(x)\cap F,\rho\right)\cover\left(B_\delta(x)\cap G,\rho\right)\\
  &\geq C^2(\delta/\rho)^{2(s-\epsilon)};
\end{align*}
it follows that $\dim_A(F\times F)\ge 2(s-\epsilon)$ for every
$\epsilon>0$, which yields (\ref{dAFF}).
\end{proof}

\section{Equi-homogeneous sets}
From Definition \ref{definition - homogeneous} we see that homogeneity
encodes the \textit{maximum} size of a local optimal cover at a
particular length-scale. However, the \textit{minimal} size of a local
optimal cover is not captured by homogeneity, and indeed this minimum
size can scale very differently, as the set described in Section \ref{section - intro homogeneity} illustrates.
\begin{example}
  For each $\alpha>0$ the set $F_{\alpha}:=
  \set{n^{-\alpha}}_{n\in\mathbb{N}}\cup\set{0}$ has Assouad dimension
  equal to $1$, so for all $\varepsilon>0$
  \[
  \sup_{x\in F_{\alpha}} \cover(B_{\delta}(x)\cap
    F_{\alpha},\rho)
  (\delta/\rho)^{-(1-\varepsilon)}
  \]
  is unbounded on $\delta,\rho$ with $0<\rho<\delta$.

  On the other hand $1\in F_{\alpha}$ is an isolated point so
  \[
  \inf_{x\in F_{\alpha}} \cover(B_{\delta}(x)\cap
    F_{\alpha},\rho) = 1
  \]
  for all $\delta,\rho$ with $0<\rho<\delta<1-2^{-\alpha}$ as
  $B_{\delta}(1)\cap F_{\alpha}=\set{1}$ for such $\delta$
  and this isolated point can be covered by a single ball of any
  radius.
\end{example}

For a totally bounded set the maximal and minimal sizes of local
optimal covers can be estimated by more elementary quantities.

\begin{lemma}\label{lemma - local cover estimates}
  For a totally bounded set $F\subset X$ and $\delta,\rho$ satisfying
  $0<\rho<\delta$
  \begin{align}
    \inf_{x\in F} \cover(B_{\delta}(x)\cap F,\rho)
	&\leq \frac{\cover(F,\rho)}{\cover(F,4\delta)}
		\label{delta rho ratio lower bound}\\
    \text{and}\qquad\sup_{x\in F} \cover(B_{\delta}(x)\cap
      F,\rho) &\geq \frac{\cover(F,\rho)}{\cover(F,\delta)}.\label{delta
      rho ratio upper bound}
  \end{align}
\end{lemma}
\begin{proof}
  Let $x_{1},\ldots,x_{\cover(F,\delta)}\in F$ be the centers of
  $\delta$-balls that form a cover of $F$. Clearly,
  \begin{align*}
    \cover(F,\rho)&\leq \sum_{j=1}^{\cover(F,\delta)}
    \cover(B_{\delta}(x_{j})\cap F,\rho) \leq
    \cover(F,\delta) \sup_{x\in F}
    \cover(B_{\delta}(x)\cap F,\rho),
  \end{align*}
  which is \eqref{delta rho ratio upper bound}.

  Next, let $\delta,\rho$ satisfy $0<\rho<\delta$ and let
  $x_{1},\ldots,x_{\pack(F,4\delta)}\in F$ be the
  centers of disjoint $4\delta$-balls. Observe that an arbitrary
  $\rho$-ball $B_{\rho}(z)$ intersects at most one of the
  balls $B_{\delta}(x_{i})$: indeed, if there exist $x,y\in
  B_{\rho}(z)$ with $x\in B_{\delta}(x_{i})$ and
  $y\in B_{\delta}(x_{j})$ with $i\neq j$ then
  \[
   d_{X}(x_{i},x_{j})\leq
   d_{X}(x_{i},x)+ d_{X}(x,z)
		+ d_{X}(z,y)+ d_{X}(y,x_{j})\leq
  2\delta + 2\rho \leq 4\delta
  \]
  and so $x_{i}\in B_{4\delta}(x_{j})$, which is a
  contradiction. Consequently, as $F$ contains the union
  $\bigcup_{j=1}^{\pack(F,4\delta)}
  B_{\delta}(x_{j})\cap F$, it follows that
  \begin{align*}
    \cover(F,\rho)&\geq
	\sum_{j=1}^{\pack(F,4\delta)}
		\cover(B_{\delta}(x_{j})\cap F,\rho)\\
    &\geq \pack(F,4\delta)
		\inf_{x\in F}\cover(B_{\delta}(x)\cap F,\rho),\\
    &\geq \cover(F,4\delta) \inf_{x\in
      F}\cover(B_{\delta}(x)\cap F,\rho)
  \end{align*}
  from \eqref{geometric inequalities}, which is precisely \eqref{delta
    rho ratio lower bound}.
\end{proof}

In contrast there is, in general, no similar elementary upper bound on the quantity $\sup_{x\in F}\cover(B_{\delta}\left(x\right)\cap F,\rho)$, the existence of which would be useful in determining the Assouad dimension of $F$. For this reason we introduce the notion of equi-homogeneity. A set is equi-homogeneous if the range of the number of sets required in the local covers is uniformly bounded at all length-scales.

\begin{definition}\label{equihom}
  We say that a set $F\subset X$ is \emph{equi-homogeneous} if for all $\delta_{0}>0$ there exist constants $M\geq 1$, and $c_{1},c_{2}>0$ such that
  \begin{align}\label{equihom equation}
    \sup_{x\in F} \cover\left(B_{\delta}(x)\cap F,\rho\right)&\leq M
    \inf_{x\in F} \cover\left(B_{c_{1}\delta}(x)\cap F,c_{2}\rho\right)
  \end{align}
  for all $\delta,\rho$ with $0<\rho<\delta\leq \delta_{0}$.
\end{definition}
As with the definition of the box-counting dimensions, it follows from the geometric inequalities \eqref{geometric inequalities} that replacing $\cover$ with the geometric quantities $\pack$ or $\diamcover$ gives an equivalent definition of equi-homogeneity. Further, note that
as $\cover(B_{\delta}\left(x\right)\cap F,\rho)$ increases with $\delta$ and decreases with $\rho$, by replacing the $c_{i}$ with $1$ if necessary we can assume that $c_{2}\leq 1\leq c_{1}$ in \eqref{equihom equation}.
If a totally bounded set $F$ is equi-homogeneous then, in addition to the lower bound \eqref{delta rho ratio upper bound}, we can find an upper bound for the maximal size of the local coverings.
\begin{corollary}\label{corollary - supremum bounds}
If $F\subset X$ is totally bounded and equi-homogeneous then for all $\delta_{0}>0$ there exist constants $M\geq 1$ and $c_{1},c_{2}>0$ with $c_{2}\leq 1\leq c_{1}$ such that
\begin{equation*}
\frac{\cover(F,\rho)}{\cover(F,\delta)}\leq \sup_{x\in F} \cover\left(B_{\delta}\left(x\right)\cap F,\rho\right) \leq M\inf_{x\in F} \cover\left(B_{c_{1}\delta}\left(x\right)\cap F,c_{2}\rho\right) \leq M\frac{\cover\left(F,c_{2}\rho\right)}{\cover\left(F,4c_{1}\delta\right)}
\end{equation*}
for all $0<\rho<\delta\leq \delta_{0}$.
\end{corollary}
\begin{proof}
Fix $\delta_{0}>0$. Assuming without loss of generality that \eqref{equihom equation} holds with $c_{2}\leq 1 \leq c_{1}$ it is clear that $\rho<\delta$ implies $c_{2}\rho\leq c_{1}\delta$. Consequently, it follows from \eqref{delta rho ratio lower bound} that
\[
\inf_{x\in F}\cover\left(B_{c_{1}\delta}\left(x\right)\cap F,c_{2}\rho\right)\leq \frac{\cover\left(F,c_{2}\rho\right)}{\cover\left(F,4c_{1}\delta\right)}
\]
for all $\delta,\rho$ with $0<\rho<\delta\leq \delta_{0}$.
The remaining inequalities are immediate from the definition of equi-homogeneity and Lemma \ref{lemma - local cover estimates}.
\end{proof}

In fact, with this bound we can precisely find the Assouad dimension of
equi-homogeneous sets provided that their box-counting dimensions are
suitably `well behaved'.

\begin{theorem}\label{theorem - equi-homogeneous Assouad equal to box}
  If a set $F\subset X$ is equi-homogeneous, attains
  both its upper and lower box-counting dimensions in the sense of (\ref{attains}), and
  $\dim_{\rm LB}F=\dim_{\rm B}F$, then 
  \[
  \dim_{\rm A}F=\dim_{\rm B}F=\dim_{\rm LB}F.
  \]
\end{theorem}

\begin{proof}
Fix $\delta_{0}>0$. As $F$ attains both its upper and lower box-counting dimensions and
  these dimensions are equal it is clear from
  \eqref{attains} that there exists a constant $C\geq 1$ such that
  \begin{align}
    \frac{1}{C}\delta^{-\dim_{\rm B}F}&\leq \cover(F,\delta)\leq
    C\delta^{-\dim_{\rm B}{F}} &&\text{for all}\;0<\delta\leq\delta_0\label{tight
      control}.
\intertext{Next, as $F$ is equi-homogeneous it follows from Corollary \ref{corollary - supremum bounds} that}
\sup_{x\in F}\cover\left(B_{\delta}\left(x\right)\cap F,\rho\right)&\leq M \frac{\cover\left(F,c_{2}\rho\right)}{\cover\left(F,4c_{1}\delta\right)} & &\text{for all}\;0<\rho<\delta\leq \delta_{0} \notag
\intertext{for some constants $M\geq 1$ and $c_{1},c_{2}>0$, which from \eqref{tight control}}
&\leq MC^{2}\frac{(c_{2}\rho)^{-\dim_{\rm B}F}}
		{(4c_{1}\delta)^{-\dim_{\rm B}F}}\notag\\
    &=MC^{2}(4c_{1}/c_{2})^{\dim_{\rm B}F}
		(\delta/\rho)^{\dim_{\rm B}F} \notag
  \end{align}
  for all $\delta,\rho$ with $0<\rho<\delta\leq\delta_{0}$, so the set $F$ is
  $(\dim_{\rm B}F)$-homogeneous. Consequently, $\dim_{\rm A}F\leq
  \dim_{\rm B}F$, but from \eqref{dimension inequalities} the Assouad
  dimension dominates the upper box-counting dimension so we obtain
  the equality $\dim_{\rm A}F=\dim_{\rm B}F$.
\end{proof}
The generalised Cantor sets introduced in the next section are the prototypical examples of equi-homogeneous sets, and it is precisely these sets that we use to construct examples of strict inequality in the Assouad dimension product formula.
In this construction we will determine the Assouad dimension of the product of generalised Cantor sets by applying the above theorem, which first requires us to show that the product set is equi-homogeneous. However, this immediately follows from the fact that equi-homogeneity is preserved upon taking products, which we now prove.

\begin{lemma}\label{lemma - equi-homogeneous products}
If $F\subset X$ and $G\subset Y$ are equi-homogeneous and the product space $X\times Y$ is endowed with a metric satisfying \eqref{product metric condition},
  then the product $F\times G\subset X\times Y$ is equi-homogeneous.
\end{lemma}
\begin{proof}
Fix $\delta_{0}>0$. As $F$ and $G$ are equi-homogeneous, there exist constants $M_{F},M_{G}\geq 1$ and $f_{1},f_{2},g_{1},g_{2}>0$ such that for all $0<\rho<\delta\leq\delta_{0}/m_{1}$
\begin{align}
\sup_{x\in F} \cover\left(B_{\delta}\left(x\right)\cap F,\rho\right)&\leq M_{F} \inf_{x\in F} \cover\left(B_{f_{1}\delta}\left(x\right)\cap F,f_{2}\rho\right)\notag\\
&\leq M_{F} \inf_{x\in F} \cover\left(B_{c_{1}\delta}\left(x\right)\cap F,c_{2}\rho\right)\label{prod F uniform}
\intertext{and}
\sup_{y\in G} \cover\left(B_{\delta}\left(y\right)\cap G,\rho\right)&\leq M_{G} \inf_{y\in G} \cover\left(B_{g_{1}\delta}\left(y\right)\cap G,g_{2}\rho\right)\notag\\
&\leq M_{G} \inf_{y\in G} \cover\left(B_{c_{1}\delta}\left(y\right)\cap G,c_{2}\rho\right),\label{prod G uniform}
\end{align}
where $c_{1}=\max\left(f_{1},g_{1}\right)$ and $c_{2}=\min\left(f_{2},g_{2}\right)$, and the second inequalities follow from the monotonicity of $\cover\left(\cdot,\rho\right)$ and the fact that $\cover\left(A,\cdot\right)$ is non-increasing.

Now, from Lemma \ref{lemma - Assouad product geometry} for all $0<\rho<\delta\leq \delta_{0}$
  \begin{align*}
    N_{F\times G}\left(\delta,\rho\right)&:=\sup_{\vect{x}\in F\times G} \cover\left(B_{\delta}(\vect{x})\cap \left(F\times G\right),\rho\right) \notag\\
    &\leq \left[\sup_{x\in F}\cover\left(B_{\delta/m_{1}}(x)\cap
        F,\rho/m_{2}\right)\right] \left[\sup_{y\in
        G}\cover\left(B_{\delta/m_{1}}(y)\cap
        G,\rho/m_{2}\right)\right]
        \end{align*}
as taking suprema is submultiplicative. Since $0<\rho/m_{2}<\delta/m_{1}\leq \delta_{0}/m_{1}$ it follows from \eqref{prod F uniform} and \eqref{prod G uniform} that $N_{F\times G}\left(\delta,\rho\right)$ is bounded above by
\begin{align*}
&\left[M_{F}\inf_{x\in F}\cover\left(B_{c_{1}\delta/m_{1}}(x)\cap F,c_{2}\rho/m_{2}\right)\right] \left[M_{G}\inf_{y\in G}\cover\left(B_{c_{1}\delta/m_{1}}(y)\cap G,c_{2}\rho/m_{2}\right)\right] \\
    &\ \leq M_{F}M_{G}\inf_{\left(x,y\right)\in F\times G}  \cover\left(B_{c_{1}\delta/m_{1}}(x)\cap F c_{2}\rho/m_{2}\right)\cover\left(B_{c_{1}\delta/m_{1}}(y)\cap G,c_{2}\rho/m_{2}\right),
\end{align*}
as taking infima is supermultiplicative. Again applying Lemma \ref{lemma - Assouad product geometry} we obtain the upper bound
\[
N_{F\times G}\left(\delta,\rho\right)    \leq M_{F}M_{G}\inf_{\vect{x}\in F\times G} \cover\left(B_{\tfrac{c_{1}m_{2}}{m_{1}}\delta}(\vect{x})\cap\left(F\times G\right),\tfrac{c_{2}m_{1}}{4m_{2}}\rho\right)
\]
  for all $0<\rho<\delta\leq \delta_{0}$ and as $\delta_{0}>0$ was arbitrary we conclude that $F\times G$ is equi-homogeneous.
\end{proof}

\section{Generalised Cantor Sets}\label{section - cantor sets}
A generalised Cantor set is a variation of the well known Cantor
middle third set that permits the proportion removed from each
interval to vary throughout the iterative process. Formally, for
$\lambda \in \left(0,1/2\right)$ we define the application of the
generator $\gen_{\lambda}$ to a disjoint set of compact intervals
$\mathcal{I}$ as the procedure in which the open middle $1-2\lambda$
proportion of each interval is removed. It is easy to see that if
$\mathcal{I}$ consists of $k$ disjoint intervals of length $L$ then
$\gen_{\lambda} \mathcal{I}$ consists of $2k$ disjoint intervals of
length $\lambda L$.

\begin{definition}
  Let $\set{\lambda_{i}}_{i\in\mathbb{N}}$ be a sequence with
  $\lambda_{i}\in(0,1/2)$ for all $i\in\mathbb{N}$, let
  $C_{0}=\left[0,1\right]$ and iteratively define the sets
  \begin{align*}
    C_{n}:= \gen_{\lambda_{n}}C_{n-1} \qquad\forall
    n\in\mathbb{N}.
  \end{align*}
  The generalised Cantor set $C$ generated from the sequence
  $\set{\lambda_{i}}_{i\in\mathbb{N}}$ is defined by
  \begin{align*}
    C:= \bigcap_{n=0}^{\infty} C_{n}.
  \end{align*}
\end{definition}
Observe that each intermediary set $C_{n}$ consists of $2^{n}$
disjoint intervals $I_{n}^{1},\ldots,I_{n}^{2^{n}}$ of length
$L_{n}:=\prod_{i=1}^{n}\lambda_{i}$, which we order by increasing left endpoint. Further, observe that the generalised Cantor
set $C$ can be written as the union of the disjoint sets
$I_{n}^{j}\cap C$ for $j=1,\ldots, 2^{n}$, which are identical up to a
translation.

In the remainder we adopt the geometric quantity $\diamcover\left(C,\delta\right)$, which we recall is the minimal cover by sets of {\it diameter} $\delta$ as our primary measure since it is convenient to cover generalised Cantor sets by collections of intervals of a fixed length, and this avoids the factor of $1/2$ that would occur if we used covers by $\delta$-balls.

It is not difficult to determine that for $\delta$ in the range
$L_{n}\leq \delta < L_{n-1}$ the minimum number of sets of diameter
$\delta$ required to cover $C$ satisfies
\begin{equation}\label{cover inequality}
  2^{n-1} \leq \diamcover(C,\delta) \leq 2^{n}
\end{equation}
(see, for example, \cite{RobinsonSharples13RAEX}.)
From this bound we can determine the upper and lower box-counting
dimensions of $C$ from the sequence
$\set{\lambda_{i}}_{i\in\mathbb{N}}$.

\begin{lemma}\label{lemma - boxcounting dimension equalities}
  Let $C$ be the generalised Cantor set generated from the sequence
  $\set{\lambda_{i}}_{i\in\mathbb{N}}$ with
  $\lambda_{i}\in(0,1/2)$. The lower and upper box-counting dimensions
  of $C$ satisfy
  \begin{align}
    \dim_{\rm LB}C &= \liminf_{n\in\mathbb{N}} \frac{n\log 2}{-\sum_{i=1}^{n}\log\lambda_{i}}\label{lower boxcounting dimension inequality}\\
    \text{and}\qquad\dim_{\rm B}C &= \limsup_{n\in\mathbb{N}} \frac{n\log
      2}{-\sum_{i=1}^{n}\log\lambda_{i}}.\label{upper boxcounting
      dimension inequality}
  \end{align}
\end{lemma}
\begin{proof}
  For $\delta$ in the range $L_{n} \leq \delta < L_{n-1}$ the
  cover estimates \eqref{cover inequality} yield
  \begin{align}
    \frac{\left(n-1\right)\log 2}{-\log L_{n}}\leq \frac{\log
      \diamcover(C,\delta)}{-\log \delta} &\leq \frac{n\log 2}{-\log
      L_{n-1}}\notag \intertext{from which we
      derive} \frac{n\log 2}{-\log L_{n}} - \frac{\log 2}{-\log
      L_{n}}\leq \frac{\log \diamcover(C,\delta)}{-\log \delta} &\leq
    \frac{\left(n-1\right)\log 2}{-\log L_{n-1}} + \frac{\log 2}{-\log
      L_{n-1}}.\label{box-counting inequality}
  \end{align}
  Taking limits as $\delta\rightarrow 0$, it is clear that
  $n\rightarrow \infty$ and $1/-\log L_{n}\rightarrow 0$ so taking the
  limit inferior of \eqref{box-counting inequality} we obtain
  \begin{align}
    \liminf_{n\in\mathbb{N}}\frac{n\log 2}{-\log L_{n}} \leq
    \liminf_{\delta\rightarrow 0+} \frac{\log \diamcover(C,\delta)}{-\log
      \delta} &\leq \liminf_{n\in\mathbb{N}}\frac{\left(n-1\right)\log
      2}{-\log L_{n-1}}\label{liminf inequality}
  \end{align}
  and as the upper and lower bounds of \eqref{liminf inequality} are
  equal we conclude that
  \begin{align*}
    \liminf_{\delta\rightarrow 0+} \frac{\log \diamcover(C,\delta)}{-\log
      \delta} &=\liminf_{n\in\mathbb{N}}\frac{n\log 2}{-\log L_{n}}\\
    &=\liminf_{n\in\mathbb{N}}\frac{n\log 2}{-\sum_{i=1}^{n}\log
      \lambda_{i}},
  \end{align*}
  which is precisely \eqref{lower boxcounting dimension
    inequality}. The upper box-counting dimension equality
  \eqref{upper boxcounting dimension inequality} follows similarly
  after taking the limit superior of \eqref{box-counting inequality}.
\end{proof}
This relationship is particularly pleasing as
$\frac{n}{\sum_{i=1}^{n} \log \lambda_{i}}=\frac{1}{\log a_{n}}$ where
$a_{n}$ is nothing more than the geometric mean of the partial
sequence $\lambda_{1},\ldots,\lambda_{n}$.

In the remainder of this section we prove that generalised Cantor sets are equi-homogeneous and use this result to determine the Assouad dimension for a certain class of generalised Cantor sets. This class includes the sets considered in Section \ref{section - construction} which, in particular, will have Assouad dimension strictly greater than their upper box-counting dimension.

To this end we first consider the minimal covers
of subintervals of $C_{n}$. The following two properties of the sets
$C_n$ are almost immediate from the construction:
\begin{itemize}
\item[(i)] for each $j=1,\ldots,2^{n}$ the subinterval $I_{n}^{j}$ of
  $C_{n}$ satisfies
  \begin{equation}\label{intervals uniform}
    \diamcover(I_{n}^{j}\cap C,\rho)=\diamcover(I_{n}^{1}\cap C,\rho);
  \end{equation}
\item[(ii)] each subinterval $I_{n-1}$ of $C_{n-1}$ satisfies
  \begin{equation}\label{subinterval bound}
    \diamcover(I_{n-1}\cap C,\rho)\leq 2 \diamcover(I_{n}^{1}\cap C,\rho)
  \end{equation}
  for all $\rho$ in the range $0<\rho<L_{n-1}$.
\end{itemize}
For the second of these, notice that by construction
$I_{n-1}\cap C=\left(I_n^i\cap C\right)\cup \left(I_n^{i+1}\cap C\right)$ for some $i$, and so
\[
    \diamcover(I_{n-1}\cap C,\rho)\le \diamcover(I_n^i\cap C,\rho) + \diamcover(I_n^{i+1}\cap C,\rho)
    = 2\diamcover(I_n\cap C,\rho).
\]

\begin{lemma}\label{lemma - Cantor sets are equi-homogeneous}
  Generalised Cantor sets are equi-homogeneous.
\end{lemma}
\begin{proof}
  Let $C$ be the generalised Cantor set generated from the sequence
  $\set{\lambda_{i}}_{i\in\mathbb{N}}$ with
  $\lambda_{i}\in(0,1/2)$. Let $x\in C$ be arbitrary, and fix $\delta$
  in the range $L_{n}\leq \delta < L_{n-1}$. As $x\in C \subset C_{n}$
  and $\delta\geq L_{n}$ the ball $B_{\delta}\left(x\right)$ contains
  at least one subinterval $I_{n}^{j}$ of $C_{n}$, so
  \begin{align}\label{interval in ball}
    I_{n}^{j}\cap C \subset B_{\delta}\left(x\right)\cap C.
  \end{align}
  Further, as $\delta<L_{n-1}$, the ball $B_{\delta}\left(x\right)$
  intersects at most three subintervals of $C_{n-1}$, say
  $I_{n-1}^{k},I_{n-1}^{k+1},I_{n-1}^{k+2}$ for some $k$, so
  \begin{align}\label{ball in two intervals}
    B_{\delta}\left(x\right)\cap C\subset \left(I_{n-1}^{k}\cap
      C\right)\cup\left(I_{n-1}^{k+1}\cap
      C\right)\cup\left(I_{n-1}^{k+2}\cap C\right).
  \end{align}
 From the inclusions \eqref{interval in ball} and \eqref{ball in two
    intervals}, the monotonicity and subadditivity of $\diamcover(\cdot,\rho)$,
  and (\ref{intervals uniform}), we derive
  \begin{align}
    \diamcover(I_{n}^{1}\cap C,\rho)\leq \diamcover(B_{\delta}\left(x)\cap C,\rho\right)
    &\leq 3\diamcover(I_{n-1}^{1}\cap C,\rho) &\text{for all}\ \rho>0\label{pre Assouad quantity bound}
  \end{align}
  for $\delta$ in the range $L_{n}\leq \delta < L_{n-1}$. Restricting
  $\rho$ to the range $0<\rho<\delta<L_{n-1}$ we apply
  (\ref{subinterval bound}) to conclude that
  \begin{align*}
    \diamcover(I_{n}^{1}\cap C,\rho)\leq \diamcover(B_{\delta}\left(x)\cap C,\rho\right)
    &\leq 6\diamcover(I_{n}^{1}\cap C,\rho)
  \end{align*}
  for all $\delta$ in the range $L_{n}\leq \delta < L_{n-1}$, all
  $\rho$ in the range $0<\rho<\delta$ and all $x\in C$.
  Consequently,
  \begin{align}\label{Cantor equihom diameter}
    \sup_{x\in C} \diamcover(B_{\delta}\left(x)\cap C,\rho\right) \leq 6
    \inf_{x\in C} \diamcover(B_{\delta}\left(x)\cap C,\rho\right)
  \end{align}
  for all $\delta,\rho$ satisfying $0<\rho<\delta$ so we conclude that $C$ is equi-homogeneous.
\end{proof}

Generalised Cantor sets, therefore, satisfy the hypotheses of Corollary \ref{corollary - supremum bounds}. We will use this fact in the following lemmas to find bounds on the Assouad dimension for a particular class of generalised Cantor sets, which will be useful in the remainder.

\begin{lemma}\label{lemma - Assouad upper bound}
If there exists a $\lambda$ such that $\lambda_{i}\leq \lambda$ for all $i\in\mathbb{N}$ then $\dim_{\rm A}C\leq-\tfrac{\log 2}{\log \lambda}$.
\end{lemma}
\begin{proof}
If $\lambda\geq 1/2$ then there is nothing to prove as trivially $\dim_{\rm A}C\leq 1$, so assume $\lambda\in \left(0,1/2\right)$.
Consider $\delta,\rho$ with $\rho<\delta$ in the ranges $L_{n}\leq \delta <L_{n-1}$ and $L_{n+m}\leq \rho < L_{n+m-1}$ for some $n\in\mathbb{N}$ and $m\in\mathbb{N}\cup\set{0}$. Observe that for $m\geq 2$
\begin{align}
\delta/\rho \geq L_{n}/L_{n+m-1} = \prod_{i=n+1}^{n+m-1} \lambda_{i}^{-1} \geq \lambda^{-\left(m-1\right)},\label{bounded generators}
\end{align}
which, as $\delta/\rho>1$, also holds when $m=0,1$.
Now, it follows from \eqref{bounded generators} and the cover estimates \eqref{cover inequality} that the ratio
\begin{align}\label{Cantor ratio upper bound}
\frac{\diamcover\left(C,\rho\right)}{\diamcover\left(C,\delta\right)}\leq \frac{2^{n+m}}{2^{m-1}}=4\cdot 2^{m-1}=4\left(\lambda^{-\left(1-m\right)}\right)^{-\frac{\log 2}{\log \lambda}}\leq 4 \left(\delta/\rho\right)^{-\frac{\log 2}{\log \lambda}},
\end{align}
which, as $n$ and $m$ were arbitrary, holds for all $\delta,\rho$ with $0<\rho<\delta$.

Next, as $C$ is equi-homogeneous it follows from Corollary \ref{corollary - supremum bounds} and the geometric inequalities \eqref{geometric inequalities} that for all $\delta_{0}>0$ there exist constants $M\geq 1$, and $c_{1},c_{2}>0$ with $c_{2}\leq 1\leq c_{1}$ such that
\begin{align*}
\sup_{x\in C}\diamcover\left(B_{\delta}\left(x\right)\cap C,\rho\right)\leq \sup_{x\in C} \cover\left(B_{\delta}\left(x\right)\cap C,\rho/2\right)&\leq M \frac{\cover\left(C,c_{2}\rho/2\right)}{\cover\left(C,4c_{1}\delta\right)}\\
&\leq M \frac{\diamcover\left(C,c_{2}\rho/4\right)}{\diamcover\left(C,8c_{1}\delta\right)}
\end{align*}
for $0<\rho<\delta\leq\delta_{0}$. As $\rho<\delta$ implies that $c_{2}\rho/4<8c_{1}\delta$ it follows from \eqref{Cantor ratio upper bound} that
\begin{align*}
\sup_{x\in C}\diamcover\left(B_{\delta}\left(x\right)\cap C,\rho\right)&\leq M 4 \left(8c_{1}\delta/\left(c_{2}\rho/4\right)\right)^{-\frac{\log 2}{\log \lambda}}
= 4M\left(32c_{1}/c_{2}\right)^{-\frac{\log 2}{\log \lambda}}\left(\delta/\rho\right)^{-\frac{\log 2}{\log \lambda}}
\end{align*}
for all $\delta,\rho$ with $0<\rho<\delta\leq \delta_{0}$. This is precisely the claim that the generalised Cantor set $C$ is $\left(-\log 2/\log \lambda\right)$-homogeneous, so we conclude that $\dim_{\rm A}C\leq -\tfrac{\log 2}{\log \lambda}$.
\end{proof}

We can also obtain a similar lower bound.
\begin{lemma}\label{lemma - Assouad lower bound}
If there exists a $\lambda\in\left(0,1/2\right)$ and a sequence $\set{j_{m}}_{m\in\mathbb{N}}$ such that for each $m\in\mathbb{N}$ the $m$ consecutive generators $\lambda_{j_{m}+1},\ldots, \lambda_{j_{m}+m}\geq \lambda$ then $\dim_{\rm A}C\geq -\tfrac{\log 2}{\log \lambda}$.
\end{lemma}
\begin{proof}
Let $\varepsilon>0$ and suppose for a contradiction that there exist $\delta_{0},\eta>0$ such that
\begin{align}\label{Assouad epsilon bound}
\sup_{x\in C} \diamcover\left(B_{\delta}\left(x\right)\cap C,\rho\right)\leq \eta\left(\delta/\rho\right)^{-\tfrac{\log 2}{\log \lambda}-\varepsilon}\quad \text{for all}\quad 0<\rho<\delta\leq\delta_{0}.
\end{align}
First observe that from Lemma \ref{lemma - local cover estimates} and the geometric inequalities \eqref{geometric inequalities} we derive
\begin{align}
\sup_{x\in C} \diamcover\left(B_{2\delta}\left(x\right)\cap C,\rho/4\right) \geq \sup_{x\in C} \cover\left(B_{2\delta}\left(x\right)\cap C,\rho/2\right)
&\geq \frac{\cover\left(C,\rho/2\right)}{\cover\left(C,2\delta\right)}\notag\\
&\geq \frac{\diamcover\left(C,\rho\right)}{\diamcover\left(C,\delta\right)}\label{ratio in D}
\end{align}
for $\delta,\rho$ with $0<\rho<\delta$.

Now, let $\delta_{m}=L_{j_{m}}$ and $\rho_{m}=L_{j_{m}+m}$, and observe that $\rho_{m}<\delta_{m}$ for all $m$, and that $\delta_{m}\rightarrow 0$ as $m\rightarrow\infty$. Moreover,
\begin{align}
\delta_{m}/\rho_{m}&=L_{j_{m}}/L_{j_{m}+m}=\prod_{i=1}^{m}\lambda_{j_{m}+i}^{-1}\leq \lambda^{-m}\label{generator consecutive bound}
\end{align}
by assumption and similarly $\delta_{m}/\rho_{m}\geq 2^{m}$ as $\lambda_{i}<1/2$ for all $i\in\mathbb{N}$.
It follows from \eqref{ratio in D} and the cover estimates \eqref{cover inequality} that
\begin{align}
\sup_{x\in C}\diamcover\left(B_{2\delta_{m}}\left(x\right)\cap C,\rho_{m}/4\right)\geq \frac{\diamcover\left(C,\rho_{m}\right)}{\diamcover\left(C,\delta_{m}\right)}\geq \frac{2^{j_{m}+m-1}}{2^{j_{m}}}=\frac{1}{2}2^{m}&=\frac{1}{2}\left(\lambda^{-m}\right)^{-\frac{\log 2}{\log \lambda}}\notag\\
&\geq \frac{1}{2}\left(\delta_{m}/\rho_{m}\right)^{-\frac{\log 2}{\log \lambda}}
\end{align}
from \eqref{generator consecutive bound}. Consequently, with the assumption \eqref{Assouad epsilon bound} it follows that
\begin{align*}
\frac{1}{2}\left(\delta_{m}/\rho_{m}\right)^{-\tfrac{\log 2}{\log \lambda}} \leq \sup_{x\in C}\diamcover\left(B_{2\delta_{m}}\left(x\right)\cap C,\rho_{m}/4\right)\leq \eta \left(2\delta_{m}/\left(\rho_{m}/4\right)\right)^{-\tfrac{\log 2}{\log \lambda}-\varepsilon}
\end{align*}
for $m$ sufficiently large that $2\delta_{m}<\delta_{0}$. Rearranging, and recalling that $\delta_{m}/\rho_{m}\geq 2^{m}$, we obtain
\begin{align*}
2^{m}\leq \delta_{m}/\rho_{m}\leq \left(2\eta8^{-\frac{\log 2}{\log \lambda}-\varepsilon}\right)^{\frac{1}{\varepsilon}}
\end{align*}
for all $m$ sufficiently large, which is a contradiction.

We conclude that $C$ is not $\left(-\log 2/\log \lambda - \varepsilon\right)$-homogeneous for any $\varepsilon>0$, so $\dim_{\rm A}C\geq -\tfrac{\log 2}{\log \lambda}$.
\end{proof}

\section{Strict inequality in the two product formulae}\label{section - construction}

In this section we provide a method for constructing two generalised
Cantor sets $C$ and $D$ so that the Assouad dimensions of these sets
and their product satisfy
\begin{align*}
\dim_{\rm A}C=\dim_{\rm A}D=\dim_{\rm A}\left(C\times D\right)=\alpha
\end{align*}
for $\alpha\in\left(0,1\right)$. In particular for these sets the Assouad dimension product inequality \eqref{Assouad product formula} is strict and maximal in the sense that the sum $\dim_{\rm A}C+\dim_{\rm A}D$ takes the maximal value $2\dim_{\rm A}\left(C\times D\right)$.

This task is significantly simplified using the results of the
previous sections that relate the Assouad dimension to the more
manageable box-counting dimensions.  In essence we construct these
sets so that the significant length-scales are common to both sets,
which is similar in approach to the compatible generalised Cantor sets
of Robinson and Sharples \cite{RobinsonSharples13RAEX}.

Let $q\in(0,\frac{1}{2})$ and let $a=\set{a_{i}}$ be a sequence of positive integers. We define two generalised Cantor sets $C$ and
$D$ via the respective sequences $\set{\lambda_{i}}$ and
$\set{\mu_{i}}$ defined by
\begin{align*}
  \lambda_{i} &:=
  \begin{cases}
    q^{a_{2k}+1} & i=n_{k}\quad\text{for some}\;k\in\mathbb{N}\\
    q & \text{otherwise}
  \end{cases}\\
  \mu_{i} &:=
  \begin{cases}
    q^{a_{2k+1}+1} & i=m_{k}\quad\text{for some}\;k\in\mathbb{N}\\
    q & \text{otherwise},
  \end{cases}
\end{align*}
where $n_{k}=\sum_{j=1}^{k}a_{2j-1}$ and
$m_{k}=a_{1}+\sum_{j=1}^{k}a_{2j}$. For brevity we say that the pair of sets $\left(C,D\right)$ is generated by $\left(q,a\right)$, and we denote the partial sum
$s_{k}=\sum_{i=1}^{k}a_{i}$. Essentially, the sequences of generators
$\lambda_{i}$ and $\mu_{i}$ are chosen so that, when $\delta$ is restricted to the
range $\left[q^{s_{k+1}},q^{s_{k}}\right]$, one of the functions
$\diamcover(C,\delta)$ or $\diamcover(D,\delta)$ scales like $\delta^{-\log 2/\log q}$
while the other is essentially constant, and such that these roles
alternate as $k$ increases. While the growth of the individual
functions $\diamcover(C,\delta)$ and $\diamcover(D,\delta)$ fluctuates with $\delta$,
the product $\diamcover(C,\delta)\diamcover(D,\delta)$ scales like $\delta^{-\log 2/\log
  q}$ for all $\delta$.

\begin{theorem}\label{productset}
 Let the pair of generalised Cantor sets $C$ and $D$ be generated by $\left(q,a\right)$. For all $\delta_0>0$ there exists a constant $\eta>0$ such that
  \begin{align}\label{product set control}
    \eta^{-1}\delta^{-\frac{\log 2}{\log q}}&\leq \diamcover(C\times
    D,\delta)\leq \eta\, \delta^{-\frac{\log 2}{\log q}}
    &\text{for all}\ 0<\delta<\delta_0,
  \end{align}
  so that in particular
\[
\dim_{\rm LB}\left(C\times D\right)=\dim_{\rm B}\left(C\times
  D\right)=\dim_A(C\times D)=-\log 2/\log q.
\]
\end{theorem}

\begin{proof}
  Using the terminology of the previous section, the intermediary sets
  $C_{n}$ and $D_{n}$ consist of $2^{n}$ intervals of length
  $L_{n}:= \prod_{i=1}^{n}\lambda_{i}$ and $M_{n}:=
  \prod_{i=1}^{n}\mu_{i}$ respectively.

  We first consider the generalised Cantor set $C$. For
  $n\in\mathbb{N}$ in the range $n_{k}\leq n < n_{k+1}$ all except $k$
  of the $\lambda_{1},\ldots,\lambda_{n}$ are equal to $q$, so
\[
L_{n}=q^{n-k}\prod_{i=1}^{k}q^{a_{2i}+1}=q^n\prod_{i=1}^k q^{a_{2i}}.
\]
Taking logarithms for clarity, we derive
\begin{align*}
  \frac{\log L_{n}}{\log q}&=n-k + \sum_{i=1}^{k}\left(a_{2i}+1\right) = n+\sum_{i=1}^{k}a_{2i} \\
  &= n-\sum_{i=1}^{k}a_{2i-1} +\sum_{i=1}^{k}a_{2i-1} +
  \sum_{i=1}^{k}a_{2i}= n-n_{k}+s_{2k},
\end{align*}
so
\begin{align}
  L_{n}&=q^{n-n_{k}+s_{2k}},& n_{k}&\leq n < n_{k+1},\label{Ln explicit}
\intertext{and, in particular, $L_{n_{k}}=q^{s_{2k}}$ and
$L_{n_{k+1}-1}=q^{s_{2k+1}-1}$. Observe that for $n_{k}\leq n <
n_{k+1}$ the length $L_{n}$ has range
$\left[L_{n_{k+1}-1},L_{n_{k}}\right]=\left[q^{s_{2k+1}-1},q^{s_{2k}}\right]$,
so inverting the relationship \eqref{Ln explicit}, we derive}
q^j&=L_{n_{k}+j-s_{2k}},& s_{2k}&\le j\le s_{2k+1}-1,\notag
\intertext{so, from the cover estimates \eqref{cover inequality},}
  2^{n_{k}+j-s_{2k}-1}&\leq \diamcover(C,q^j)\leq 2^{n_{k}+j-s_{2k}},&  s_{2k}&\le j\le,s_{2k+1}-1. \label{C variable range}
\intertext{Further, observe that \eqref{Ln explicit} yields
$L_{n_{k+1}}=q^{s_{2k+2}}$ and $L_{n_{k+1}-1}=q^{s_{2k+1}-1}$, so}
L_{n_{k+1}}&\leq q^j\leq L_{n_{k+1}-1},& s_{2k+1}-1&\le j\le
s_{2k+2},\notag
\intertext{from which, with the cover estimates \eqref{cover inequality}, we conclude that}
  2^{n_{k+1}-2}&\leq \diamcover(C,q^j)\leq 2^{n_{k+1}},& s_{2k+1}-1&\le j\le,s_{2k+2}.\label{C constant range}
\intertext{A very similar argument shows that for the set $D$ the bounds}
  2^{m_{k}+j-s_{2k+1}-1}&\leq \diamcover(D,q^j) \leq 2^{m_{k}+j-s_{2k+1}},& s_{2k+1}&\le j\le s_{2k+2}-1.\label{D variable range}
\intertext{and}
  2^{m_{k}-2}&\leq \diamcover(D,q^j)\leq 2^{m_{k}},& s_{2k}-1&\le j\le s_{2k+1}.\label{D constant range}
\end{align}
hold.

Now, taking the product of \eqref{C constant range} and \eqref{D
  variable range} we obtain
\begin{equation}
  2^{n_{k+1}+m_{k}+j-s_{2k+1}-3}\leq \diamcover(C,q^j)\diamcover(D,q^j) \leq 2^{n_{k+1}+m_{k}+j-s_{2k+1}}\label{first product}
\end{equation}
for $s_{2k+1}\le j\le s_{2k+2}-1$, and multiplying \eqref{C variable
  range} with \eqref{D constant range} yields
\begin{equation}
  2^{n_{k}+m_{k}+j-s_{2k}-3}\leq \diamcover(C,q^j)\diamcover(D,q^j) \leq 2^{n_{k}+m_{k}+j-s_{2k}} \label{second product}
\end{equation}
for $s_{2k}\le j\le s_{2k+1}-1$.  Finally, since
$n_{k}+m_{k}=s_{2k}+a_{1}$ and $n_{k+1}+m_{k}=s_{2k+1}+a_{1}$, the
bounds \eqref{first product} and \eqref{second product} are precisely
\begin{align}\label{q alpha product bounds}
  2^{j+a_{1}-3}&\leq \diamcover(C,q^j)\diamcover(D,q^j)\leq 2^{j+a_{1}}
\end{align}
for $s_{2k}\le j\le s_{2k+2}-1$ and, as $k\in\mathbb{N}$ was
arbitrary, we see that \eqref{q alpha product bounds} holds for all
$j\geq s_{2}$. It follows that
\begin{equation*}
  2^{a_{1}-3}\delta^{\frac{\log 2}{\log q}}\leq \diamcover(C,\delta)\diamcover(D,\delta)\leq 2^{a_{1}}\delta^{\frac{\log 2}{\log q}},\qquad \text{for all}\ 0<\delta<q^{s_{2}}.
\end{equation*}

Finally, recall from the product inequality \eqref{product geometry
  inequality in N} and the geometric relationships \eqref{geometric inequalities} that for all $\delta>0$
\[
\diamcover(C,8\delta)\diamcover(D,8\delta)\leq \diamcover(C\times D,\delta)\leq
\diamcover(C,\delta/2\sqrt2)\diamcover(D,\delta/2\sqrt2),
\]
whence
\begin{align}\label{final product bounds}
\eta^{-1}\delta^{\frac{\log 2}{\log q}}\leq \diamcover(C\times D,\delta) \leq
\eta\delta^{\frac{\log 2}{\log q}},\qquad \text{for all}\ 0<\delta <q^{s_{2}}.
\end{align}
If $\delta_{0}<q^{s_{2}}$ then \eqref{final product bounds} implies \eqref{product set
  control}, otherwise observe that for $\delta$ in the range $q^{s_{2}}\leq \delta \leq \delta_{0}$ trivially
\[
\diamcover\left(C\times D,\delta_{0}\right)\left(q^{s_{2}}\right)^{-\frac{\log 2}{\log q}}\delta^{\frac{\log 2}{\log q}}\leq \diamcover\left(C\times D,\delta\right)\leq \diamcover\left(C\times D,q^{s_{2}}\right) \delta_{0}^{-\frac{\log 2}{\log q}} \delta^{\frac{\log 2}{\log q}}
\]
which, together with \eqref{final product bounds}, yields \eqref{product set control}.

This immediately shows that the upper and lower box-counting
dimensions coincide, are attained, and are equal to $-\log 2/\log q$. The same
expression for the Assouad dimension then follows using Theorem \ref{theorem - equi-homogeneous Assouad equal to box} and the fact that the product set $C\times D$ is equi-homongeneous, being the product of two equi-homogeneous sets $C$ and $D$ (Lemmas \ref{lemma - equi-homogeneous products} and \ref{lemma - Cantor sets are equi-homogeneous}).
\end{proof}

\begin{theorem}\label{boxdims}
Let the pair of generalised Cantor sets $C$ and $D$ be generated by $\left(q,a\right)$.
  The upper box-counting dimensions of $C$ and $D$ are given by
  \begin{align}
    \dim_{\rm B}C&=-\left(\frac{\log 2}{\log
        q}\right)\limsup_{k\in\mathbb{N}}\frac{\sum_{j=1}^{k}a_{2j-1}}{\sum_{j=1}^{2k-1}a_{i}}\label{dimB
      C bounds}\\ \text{and} \qquad \dim_{\rm B}D &=-\left(\frac{\log 2}{\log
        q}\right)\limsup_{k\in\mathbb{N}}\frac{\sum_{j=1}^{k}a_{2j}}{\sum_{i=1}^{2k}a_{i}}\notag
  \end{align}
  respectively.
\end{theorem}

\begin{proof}
  Recall from Lemma \ref{lemma - boxcounting dimension equalities}
  that the upper-box counting dimensions of $C$ and $D$ are given by
  \begin{align*}
    \dim_{\rm B}C&=\limsup_{n\in\mathbb{N}}\frac{n\log 2}{-\log L_{n}}
    &\text{and}\qquad \dim_{\rm B}D&=\limsup_{n\in\mathbb{N}}\frac{n\log
      2}{-\log M_{n}}.
  \end{align*}
  We first consider the generalised Cantor set $C$. For
  $n\in\mathbb{N}$ in the range $n_{k}\leq n <n_{k+1}$ we obtain from
  \eqref{Ln explicit}
  \begin{align*}
    \frac{n\log 2}{-\log L_{n}} = \frac{n \log 2}{-\left(n-n_{k}+s_{2k}\right)\log q} &\leq -\left(\frac{\log 2}{\log q}\right)\frac{n_{k+1}}{s_{2k+1}}\\
    &= -\left(\frac{\log 2}{\log
        q}\right)\frac{\sum_{i=1}^{k+1}a_{2j-1}}{\sum_{i=1}^{2k+1}a_{i}},
  \end{align*}
  where we have used the fact that $n/(a+n)$ is increasing in $n$ for
  $a>0$. Taking the limit superior as $n$ (and hence $k$) tend to
  infinity we conclude that
  \begin{align*}
    \dim_{\rm B}C&\leq -\left(\frac{\log 2}{\log
        q}\right)\limsup_{k\in\mathbb{N}}\frac{\sum_{i=1}^{k+1}a_{2j-1}}{\sum_{i=1}^{2k+1}a_{i}},
  \end{align*}
  which is the upper bound in \eqref{dimB C bounds}.  To establish the
  lower bound we consider the subsequence $n_{k+1}-1$ and recall from
  \eqref{Ln explicit} that
  $L_{n_{k+1}-1}=q^{s_{2k+1}-1}$. Consequently,
  \begin{align*}
    \frac{\left(n_{k+1}-1\right)\log 2}{-\log L_{n_{k+1}-1}} =
    \frac{\left(n_{k+1}-1\right)\log 2}{-\left(s_{2k+1}-1\right)\log
      q}=-\left(\frac{\log 2}{\log
        q}\right)\frac{\sum_{j=1}^{k+1}a_{2j-1}-1}{\sum_{i=1}^{2k+1}a_{i}-1},
  \end{align*}
  so
  \begin{align*}
    \dim_{\rm B}C =\limsup_{n\in\mathbb{N}} \frac{n\log 2}{-\log L_{n}} &\geq \limsup_{k\in\mathbb{N}}\frac{\left(n_{k+1}-1\right)\log 2}{-\log L_{n_{k+1}-1}}\\
    &= -\left(\frac{\log 2}{\log q}\right) \limsup_{k\in\mathbb{N}}
    \frac{\sum_{j=1}^{k+1}a_{2j-1}}{\sum_{i=1}^{2k+1}a_{i}}.
  \end{align*}
  Since these upper and lower bounds coincide we obtain the equality
  in (\ref{dimB C bounds}).

  The argument for $D$ follows similar lines.\end{proof}

In general the Assouad dimension dominates the upper box-counting dimension, so the above theorem provides lower bounds for the Assouad dimension of the sets $C$ and $D$.
However, using the results of the previous section, we can precisely determine the Assouad dimension of the sets $C$ and $D$ provided that the odd and even terms of the sequence $\set{a_{i}}$ respectively are unbounded.

\begin{theorem}\label{theorem - assouad dims}
If the generalised Cantor sets $\left(C,D\right)$ are generated by $\left(q,a\right)$ then
\begin{align*}
\sup\set{a_{2i-i}}=\infty\qquad &\text{implies} \qquad \dim_{\rm A}C=-\tfrac{\log 2}{\log q},\\
\text{and}\qquad \sup\set{a_{2i}}=\infty\qquad &\text{implies} \qquad \dim_{\rm A}D=-\tfrac{\log 2}{\log q}.
\end{align*}
\end{theorem}
\begin{proof}
If the sequence $a_{2j-1}$ is unbounded then there exists a subsequence $a_{2j_{k}-1}$ such that $a_{2j_{k}-1}>k$ for all $k\in\mathbb{N}$.
It follows that $n_{j_{k}-1}+k<n_{j_{k}-1}+a_{2j_{k}-1}=n_{j_{k}}$, so from the definition of the $\lambda_{i}$ we conclude that the $k$ consecutive generators
\[
\lambda_{i}=q \quad\text{for}\quad i=n_{j_{k}-1}+1,\ldots,n_{j_{k}-1}+k.
\]
It follows from Lemma \ref{lemma - Assouad lower bound} that $\dim_{\rm A}C\geq -\log 2/\log q$. The opposite inequality follows from Lemma \ref{lemma - Assouad upper bound} as $\lambda_{i}\leq q$ for all $i\in\mathbb{N}$, so we conclude that $\dim_{\rm A}C=-\log 2/\log q$.
The argument for the set $D$ follows similar lines.
\end{proof}

In summary we have constructed generalised Cantor sets $C$ and $D$
such that
\begin{align*}
\dim_B C&= -\frac{\log 2}{\log q}\limsup_{k\in\mathbb{N}}\frac{\sum_{i=1}^{k+1}a_{2j-1}}{\sum_{i=1}^{2k+1}a_{i}},\\
\dim_B D&= -\frac{\log 2}{\log
      q}\limsup_{k\in\mathbb{N}}\frac{\sum_{j=1}^{k+1}a_{2j}}{\sum_{i=1}^{2k+2}a_{i}},\\
\dim_{\rm A}\left(C\times D\right)&=\dim_B(C\times D)=
  -\frac{\log 2}{\log q},\\
  \dim_{\rm A}C&=-\frac{\log 2}{\log q} \qquad\text{if $\set{a_{2i-1}}$ is unbounded,}\\
 \text{and}\qquad \dim_{\rm A}D&=-\frac{\log 2}{\log q} \qquad\text{if $\set{a_{2i}}$ is unbounded.}
\end{align*}

By choosing the $\{a_i\}$ appropriately we can now produce generalised
Cantor sets $C$ and $D$ such that
\begin{align*}
\dim_{\rm A}C=\dim_{\rm A}D=\dim_{\rm A}\left(C\times D\right) = \dim_{\rm B}\left(C\times D\right)=\dim_{\rm LB}\left(C\times D\right),
\end{align*}
where the box-counting dimensions of these sets take arbitrary values satisfying the product formula
\begin{align*}
\dim_{\rm B}C,\dim_{\rm B}D&\leq \dim_{\rm B}\left(C\times D\right) \leq \dim_{\rm B}C+\dim_{\rm B}D,
\end{align*}
subject to the restrictions that
\begin{align}
0&<\dim_{\rm B}C,\dim_{\rm B}D<\dim_{\rm B}\left(C\times D\right)<1.\label{boxdims restrictions}
\end{align}
In particular the Assouad dimension of the product satisfies
\[
\dim_{\rm A}\left(C\times D\right)<\dim_{\rm A}C+\dim_{\rm A}D = 2\dim_{\rm A}\left(C\times D\right)
\]
so there is a strict inequality in the Assouad dimension product formula \eqref{Assouad product formula}. Further, these sets give extreme examples of strict inequality in the product formula as, in general, $\dim_{\rm A}F+\dim_{\rm A}G\leq 2\dim_{\rm A}\left(F\times G\right)$ for arbitrary sets $F,G$.

\begin{lemma}\label{lemma - construction limit}
Let $\alpha,\beta \in \left(0,1\right)$. There exist generalised Cantor sets $C$ and $D$ such that
\[
\dim_{\rm LB}C=\dim_{\rm B}C=\alpha\beta,\qquad \dim_{\rm LB}D=\dim_{\rm B}D=\alpha\left(1-\beta\right),
\]
and
\begin{align*}
\dim_{\rm LB}\left(C\times D\right)=\dim_{\rm B}\left(C\times D\right)=\dim_{\rm A}\left(C\times D\right)=\dim_{\rm A}C=\dim_{\rm A}D=\alpha.
\end{align*}
\end{lemma}
\begin{proof}
Define the sequence $a=\set{a_{i}}$ by $a_{2k-1}=\left\lceil\beta k\right\rceil$ and $a_{2k}=\left\lceil\left(1-\beta\right)k\right\rceil$ where the ceiling function $\left\lceil x \right\rceil$ is the smallest integer greater than or equal to $x$. Clearly $\beta,1-\beta>0$ so the $a_{i}$ are positive integers. Let the pair of generalised Cantor sets $C$ and $D$ be generated by $\left(2^{-1/\alpha},a\right)$, so
immediately from Theorem \ref{productset} we obtain
\begin{align*}
\dim_{\rm LB}\left(C\times D\right)=\dim_{\rm B}\left(C\times D\right)=\dim_{\rm A}\left(C\times D\right)=\alpha
\end{align*}
as required. Further, as both the odd and even terms $a_{2i-1}$ and $a_{2i}$ are unbounded we obtain $\dim_{\rm A}C=\dim_{\rm A}D=\alpha$ from Theorem \ref{theorem - assouad dims}.

Next, observe that
\begin{align*}
\tfrac{1}{2}k\left(k+1\right)\left(1-\beta\right)&\leq \sum_{j=1}^{k} a_{2j} \leq \tfrac{1}{2}k\left(k+1\right)\left(1-\beta\right) + k\\
\text{and}\qquad\tfrac{1}{2}k\left(k+1\right)\beta &\leq \sum_{j=1}^{k} a_{2j-1} \leq \tfrac{1}{2}k\left(k+1\right)\beta + k.
\end{align*}
Consequently,
\begin{align*}
\frac{\sum_{j=1}^{k}a_{2j-1}}{\sum_{j=1}^{2k-1}a_{i}} &\leq \frac{\tfrac{1}{2}k\left(k+1\right)\beta + k}{\tfrac{1}{2}k\left(k-1\right)\left(1-\beta\right) + \tfrac{1}{2}k\left(k+1\right)\beta}\\
&=\frac{\left(k+1\right)\beta +2}{k-1+2\beta}\rightarrow \beta
\intertext{as $k\rightarrow \infty$, while}
\frac{\sum_{j=1}^{k}a_{2j-1}}{\sum_{j=1}^{2k-1}a_{i}} &\geq \frac{\tfrac{1}{2}k\left(k+1\right)\beta}{\tfrac{1}{2}k\left(k-1\right)\left(1-\beta\right) + k-1 + \tfrac{1}{2}k\left(k+1\right)\beta + k}\\
&=\frac{k\left(k+1\right)\beta}{k^{2}+k-1+2k\beta}\rightarrow \beta
\end{align*}
as $k\rightarrow \infty$. It follows from Theorem \ref{boxdims} that
$\dim_{\rm B}C=\alpha\beta$ as required, and from a similar argument we obtain $\dim_{\rm B}D=\alpha\left(1-\beta\right)$.
Finally, observe that from the chain of product inequalities \eqref{product inequality chain} we obtain
\begin{align*}
\dim_{\rm LB}C+\dim_{\rm B}D=\dim_{\rm B}C+\dim_{\rm LB}D=\dim_{\rm B}\left(C\times D\right),
\end{align*}
which implies that $\dim_{\rm LB}C=\alpha\beta$ and $\dim_{\rm LB}D=\alpha\left(1-\beta\right)$.
\end{proof}

The previous lemma is a limiting case of the following more general construction, which gives independent control over the box-counting dimensions of $C$ and $D$.

\begin{lemma}\label{lemma - construction}
Let $\alpha,\beta,\gamma \in\left(0,1\right)$ be such that $\beta+\gamma>1$. There exist generalised Cantor sets $C$ and $D$ such that
\begin{align*}
\dim_{\rm LB}C&=\alpha\left(1-\gamma\right),&\dim_{\rm B}C&=\alpha\beta,\\
\dim_{\rm LB}D&=\alpha\left(1-\beta\right),& \dim_{\rm B}D&=\alpha\gamma,
\end{align*}
and
\begin{align*}
\dim_{\rm LB}\left(C\times D\right)=\dim_{\rm B}\left(C\times D\right)=\dim_{\rm A}\left(C\times D\right)=\dim_{\rm A}C=\dim_{\rm A}D=\alpha.
\end{align*}
\end{lemma}

\begin{proof}
We first observe that $\tfrac{\beta}{1-\beta},\tfrac{\gamma}{1-\gamma}>0$ and that
\begin{align}\label{betagamma}
\frac{\gamma\beta}{\left(1-\gamma\right)\left(1-\beta\right)}> 1
\end{align}
follows from $\beta+\gamma> 1$.

Now recursively define the sequence $\set{a_{i}}$ by
\begin{align}
a_{1}&=1,\notag\\
a_{2}&=\left\lceil \tfrac{\gamma}{1-\gamma} \right\rceil +1,\notag\\
a_{2k+1}&=\left\lceil \tfrac{\beta}{1-\beta} e_{k} - o_{k} \right\rceil+1,\label{odd ai def}\\
\text{and}\qquad a_{2k+2}&=\left\lceil \tfrac{\gamma}{1-\gamma} o_{k+1} - e_{k} \right\rceil+1 \label{even ai def}
\end{align}
for $k\in\mathbb{N}$, where $o_{k}=\sum_{j=1}^{k}a_{2j-1}$ and $e_{k}=\sum_{j=1}^{k}a_{2j}$ are the sums of the odd and of the even terms of $a_{i}$ respectively.
Observe that
\begin{align*}
a_{2k+2} \geq \tfrac{\gamma}{1-\gamma}o_{k+1}-e_{k}+1 &= \tfrac{\gamma}{1-\gamma}\left(a_{2k+1}+o_{k}\right)-e_{k}+1\\
&\geq \tfrac{\gamma}{1-\gamma}\left(\tfrac{\beta}{1-\beta}e_{k}+1\right)-e_{k}\\
&=\left(\tfrac{\gamma\beta}{\left(1-\gamma\right)\left(1-\beta\right)}-1\right)e_{k}+\tfrac{\gamma}{1-\gamma}
\end{align*}
and similarly
\begin{align*}
a_{2k+3} \geq \left(\tfrac{\gamma\beta}{\left(1-\gamma\right)\left(1-\beta\right)}\right)o_{k+1} + \tfrac{\beta}{1-\beta},
\end{align*}
from which, with \eqref{betagamma}, a straightforward inductive argument shows that the $a_{i}$ are positive integers with unbounded odd and even terms.

Now, let the pair of generalised Cantor sets $\left(C,D\right)$ be generated from $\left(2^{-1/\alpha},a\right)$. From Theorem \ref{productset} we obtain
\begin{align*}
\dim_{\rm LB}\left(C\times D\right)=\dim_{\rm B}\left(C\times D\right)=\dim_{\rm A}\left(C\times D\right)=\alpha
\end{align*}
and from Theorem \ref{theorem - assouad dims} that $\dim_{\rm A}C=\dim_{\rm A}D=\alpha$ as required. Further, from Theorem \ref{boxdims},
\begin{align*}
\dim_{\rm B}C&=\alpha\limsup_{k\in\mathbb{N}}\frac{\sum_{i=1}^{k+1}a_{2j-1}}{\sum_{i=1}^{2k+1}a_{i}}\notag\\
&=\alpha\limsup_{k\in\mathbb{N}}\frac{o_{k+1}}{o_{k+1}+e_{k}}=\alpha\limsup_{k\in\mathbb{N}}\frac{1}{1+\frac{e_{k}}{o_{k+1}}}
\end{align*}

and from \eqref{odd ai def} it follows that
\[
\frac{1-\beta}{\beta}\frac{e_{k}}{e_{k}+2\frac{1-\beta}{\beta}}\le\frac{e_k}{o_{k+1}}\le \frac{1-\beta}{\beta}\frac{e_{k}}{e_{k}+\frac{1-\beta}{\beta}},
\]
so we conclude that $\dim_{\rm B}C=\alpha\beta$. A similar argument using \eqref{even ai def} shows that $\dim_{\rm B}D=\alpha\gamma$.
As in Lemma \ref{lemma - construction limit} the lower box-counting dimensions are obtained from the chain of dimension inequalities \eqref{product inequality chain}.
\end{proof}

In conclusion we have demonstrated that the class of generalised Cantor sets include natural, elementary examples of sets for which the Assouad dimension product inequality is strict and maximal in the sense that the upper bound
\[
\dim_{\rm A}\left(C\times D\right)\leq\dim_{\rm A}C+\dim_{\rm A}D\leq 2\dim_{\rm A}\left(C\times D\right)
\]
is actually an equality. Further, inside this class of sets are examples that, in addition, have box-counting dimensions with arbitrary values satisfying
\begin{align*}
\dim_{\rm B}C,\dim_{\rm B}D&\leq \dim_{\rm B}\left(C\times D\right) \leq \dim_{\rm B}C+\dim_{\rm B}D,
\end{align*}
subject to the restrictions \eqref{boxdims restrictions}.

\appendix
\section{Box-counting dimensions of self-products}

The following product dimension equality is interesting, particularly
in the light of the parallel result for the Assouad dimension
presented here in Lemma \ref{Adimprod}. However, since it falls
outside the main scope of this paper we give it in this brief
appendix.

\begin{lemma}
Let $\left(X,\rmd_{X}\right)$ be a metric space and equip the product space $X\times X$ with a metric satisfying \eqref{product metric condition}.
  For all totally bounded sets $F\subset X$
  \begin{align*}
    \dim_{\rm B}\left(F\times F\right) &=2 \dim_{\rm B}F\\
    \text{and}\qquad \dim_{\rm LB}\left(F\times F\right) &= 2\dim_{\rm LB}F.
  \end{align*}
\end{lemma}
\begin{proof}
  Let $F,G\subset X$ be totally bounded sets. Recall from \eqref{product geometry inequality in N} that for all
  $\delta>0$

\begin{align*}
  \cover(F,4\delta/m_{1})\cover(G,4\delta/m_{1})\leq \cover(F\times G,\delta)\leq
  \cover(F,\delta/m_{2})\cover(G,\delta/m_{2})
\end{align*}

Consequently,
\begin{align*}
  \frac{\log \cover(F\times G,\delta)}{-\log \delta} &\leq \frac{\log \cover(F,\delta/m_{2})}{-\log \delta} + \frac{\log \cover(G,\delta/m_{2})}{-\log \delta}\notag\\
  &= \frac{\log \cover(F,\delta/m_{2})}{-\log
    \left(\delta/m_{2}\right) + \log\left(m_{2}\right)} +
  \frac{\log \cover(G,\delta/m_{2})}{-\log \left(\delta/m_{2}\right)+
    \log\left(m_{2}\right)}
  \intertext{and}
  \frac{\log \cover(F\times G,\delta)}{-\log \delta} &\geq \frac{\log \cover(F,4\delta/m_{1})}{-\log \delta} + \frac{\log \cover(G,4\delta/m_{1})}{-\log \delta}\notag\\
  &= \frac{\log \cover(F,4\delta/m_{1})}{-\log \left(4\delta/m_{1}\right)+\log\left(m_{1}/4\right) } +
  \frac{\log \cover(G,4\delta/m_{1}\delta)}{-\log \left(4\delta/m_{1}\right)+\log \left(m_{1}/4\right)
    }.
\end{align*}
These upper and lower bounds have the same limit superior and the same limit inferior as
$\delta\rightarrow 0+$, so we obtain
\begin{align}
  \limsup_{\delta\rightarrow 0+}\frac{\log \cover(F\times G,\delta)}{-\log
    \delta} &= \limsup_{\delta\rightarrow 0+} \left(\frac{\log
      \cover(F,\delta)}{-\log \delta} + \frac{\log \cover(G,\delta)}{-\log
      \delta}\right)\label{product supremum equality} \intertext{and}
  \liminf_{\delta\rightarrow 0+}\frac{\log \cover(F\times G,\delta)}{-\log
    \delta} &= \liminf_{\delta\rightarrow 0+} \left(\frac{\log
      \cover(F,\delta)}{-\log \delta} + \frac{\log \cover(G,\delta)}{-\log
      \delta}\right).\label{product infimum equality}
  \intertext{Consequently, in the case $F=G$}
  \limsup_{\delta\rightarrow 0+}\frac{\log \cover(F\times F,\delta)}{-\log
    \delta} &=2 \limsup_{\delta\rightarrow 0+} \frac{\log
    \cover(F,\delta)}{-\log \delta}\notag\\
\text{and}\qquad  \liminf_{\delta\rightarrow 0+}\frac{\log \cover(F\times F,\delta)}{-\log
    \delta} &=2 \liminf_{\delta\rightarrow 0+} \frac{\log
    \cover(F,\delta)}{-\log \delta}.\notag
\end{align}
\end{proof}
We remark that the general box-counting dimension product inequalities
follow from \eqref{product supremum equality} and \eqref{product
  infimum equality} and the fact that taking limits superior is
subadditive whilst taking limits inferior is superadditive.

\bibliographystyle{plain}
\bibliography{biblio.bib}

\end{document}